\documentclass[12pt,a4paper]{article}

\usepackage[utf8]{inputenc}
\usepackage[T1]{fontenc}
\usepackage[english]{babel}
\usepackage{lmodern}

\usepackage{amsfonts}
\usepackage{amsmath}
\usepackage{amssymb}
\usepackage{amsthm}

\usepackage{geometry}
\newtheorem{thm}{Theorem}

\newcommand{\bsa}{\boldsymbol{a}}
\newcommand{\bsb}{\boldsymbol{b}}
\newcommand{\bsx}{\boldsymbol{x}}
\newcommand{\bsy}{\boldsymbol{y}}
\newcommand{\bsk}{\boldsymbol{k}}
\newcommand{\bsgamma}{\boldsymbol{\gamma}}

\newcommand{\cH}{{\cal H}}
\newcommand{\bszero}{\boldsymbol{0}}

\newcommand{\rd}{\,\mathrm{d}}

\newcommand{\NN}{\mathbb{N}}

\newcommand{\RR}{\mathbb{R}}
\newcommand{\E}{\mathbb{E}}

\newcommand{\TODO}[1]{}
\newcommand{\EXCLUDE}[1]{}

\begin{document}

\title{Tractability of Monte Carlo integration in Hermite spaces}

\author{Christian Irrgeher\thanks{C. Irrgeher is supported by the Austrian Science Fund (FWF): Project F5509-N26, which is a part of the Special Research Program ``Quasi-Monte Carlo Methods: Theory and Applications''}}

\maketitle

\begin{abstract}
We consider multivariate integration in the randomized setting. The function spaces which we study are defined on $\RR^s$ with respect to the Gaussian measure and the functions are characterized by the decay of their Hermite coefficients. We study tractability of Monte Carlo integration and give necessary and sufficient conditions to achieve tractability.
\end{abstract}
 
 


\section{Introduction}\label{sectractability}

Study tractability of multivariate problems, like integration, in reproducing kernel Hilbert spaces goes back to the works of Hickernell \cite{H95} and Sloan and Wo\'zniakowski \cite{SW98}. Since then different notions of tractability were studied for multivariate problems in various function spaces. However, there are only a few results about tractability of multivariate integration of functions defined on unbounded domains, see e.g., \cite{KWW06, NK14, WW02}. In this paper we want to consider tractability of integration in spaces of functions defined on $\RR^s$. For that, we consider the problem of approximating integrals of the form
\begin{align}\label{eq:int}
I_s(f)=\int_{\RR^s} f(\bsx) \varphi_s(\bsx) d \bsx
\end{align}
where $\varphi_s$ denotes the density of the $s$-dimensional standard Gaussian measure, 
\begin{align*}
\varphi_s(\bsx)=\frac{1}{(2 \pi)^{s/2}} \exp\left(-\frac{\bsx \cdot \bsx}{2}\right),
\end{align*}
where ``$\cdot$'' denotes the standard inner product on $\RR^s$. Moreover, we consider integrands $f$ which belong to a reproducing kernel Hilbert space $\mathcal{H}(K)$ with norm $\| \cdot \|_K$. 

In \cite{IL14} reproducing kernel Hilbert spaces, so-called Hermite spaces, are introduced for which the problems \eqref{eq:int} are well-defined. These function spaces are defined on the $\RR^s$ with respect to the Gaussian measure and they are based on Hermite polynomials. For Hermite spaces of functions with polynomially and exponentially decaying Hermite coefficients tractability of multivariate integration was already studied in the worst case setting, see \cite{IL14} and \cite{IKLP14}.

In this paper we are interested in approximations of \eqref{eq:int} obtained by Monte Carlo (MC) integration rules which are randomized linear algorithms with equal weights and randomly chosen integration nodes. That is, we study tractability in the randomized setting, for more details see Chapter 7 in \cite{NW08}. We will proceed as in \cite{sw01} and \cite{dp05} where tractability of MC integration is studied for the Korobov space and for the Walsh space, respectively. 

The rest of the paper is structured as follows: In Section \ref{secspace} we introduce the general concept of Hermite spaces. Moreover, we present two interesting classes of Hermite spaces: Hermite spaces of functions with polynomially decaying Hermite coefficients and Hermite spaces of functions with exponentially decaying Hermite coefficients. Section \ref{sectrac} deals with tractability of MC integration in Hermite spaces and we give necessary and sufficient conditions to achieve different notions of tractability.

\section{The Hermite spaces}\label{secspace}

We start by introducing some basic facts on \textit{Hermite polynomials}. For more details on Hermite polynomials we refer to \cite{B98, s77,szeg}. For $k\in\NN_0$ the $k$th Hermite polynomial is given by 
\begin{align*}
H_k(x)=\frac{(-1)^k}{\sqrt{k!}} \exp(x^2/2) \frac{\rd^k}{\rd x^k} \exp(-x^2/2),
\end{align*}
where we follow the definition given in \cite{B98}. We remark that there are slightly different ways to introduce Hermite polynomials, see, e.g., \cite{szeg}. For $s \ge 2$, $\bsk=(k_1,\ldots,k_s)\in \NN_0^s$, and $\bsx=(x_1,\ldots,x_s)\in \RR^s$ we define $s$-dimensional Hermite polynomials by 
\begin{align*}
H_{\bsk}(\bsx)=\prod_{j=1}^s H_{k_j}(x_j).
\end{align*} 
It is well-known, see \cite{B98}, that the Hermite polynomials $\{H_{\bsk}(\bsx)\}_{\bsk \in \NN_0^s}$ form an orthonormal basis of the space $L^2(\RR^s,\varphi_s)$ of function which are square-integrable with respect to the Gaussian measure. 

Now we are going to define function spaces based on Hermite polynomials. These kind of function spaces were first introduced in \cite{IL14}. Let $r: \NN_0^s \rightarrow \RR^+$ be a summable function, i.e., $\sum_{\bsk\in\NN_0^s} r(\bsk) < \infty$. Define a kernel function 
\begin{align*}
K_{r}(\bsx,\bsy)=\sum_{\bsk \in \NN_0^s} r(\bsk) H_{\bsk}(\bsx) H_{\bsk}(\bsy)\ \ \ \ \mbox{ for }\ \ \bsx,\bsy \in \RR^s
\end{align*}
and an inner product 
\begin{align*}
\langle f,g\rangle_{K_{r}} =\sum_{\bsk \in \NN_0^s} \frac{1}{r(\bsk)} \widehat{f}(\bsk) \widehat{g}(\bsk),
\end{align*}
where $\widehat{f}(\bsk)=\int_{\RR^s} f(\bsx) H_{\bsk}(\bsx) \varphi_s(\bsx)\rd \bsx$ is the $\bsk$th \textit{Hermite coefficient} of $f$. Since $K_{r}$ is symmetric and positive semi-definite, we indeed have that $K_{r}$ is a reproducing kernel, see, e.g.\ \cite[Chapter 2.3]{dp05}. Let us denote by $\mathcal{H}(K_r)$ the reproducing kernel Hilbert space corresponding to $K_{r}$. The function space $\cH(K_r)$ is called a {\em Hermite space} and the norm in $\mathcal{H}(K_r)$ is defined in the natural way by $\| f \|_{K_{r}}^2 =\langle f , f \rangle_{K_{r}}$. More details on reproducing kernel Hilbert spaces can be found in \cite{A50}.

Note that a Hermite space $\cH(K_r)$ is fully specified by the function $r$ which regulates the decay of the Hermite coefficients of the functions belonging to $\cH(K_r)$. Roughly speaking, the faster $r$ decreases as $\bsk$ grows, the faster the Hermite coefficients of the elements of $\mathcal{H}(K_r)$ decrease. 

In this paper we deal with two important classes of Hermite spaces, namely Hermite spaces of functions with polynomially decaying Hermite coefficients and Hermite spaces of functions with exponentially decaying Hermite coefficients. Moreover, we introduce weights to the norm of these function spaces to control the influence of each coordinate.

\subsection{Hermite spaces of finite smoothness}

To define our function $r$, we first choose a weight sequence of positive real numbers, $\bsgamma=\{\gamma_j\}_{j\in\NN}$ with $\gamma_j>0$, where we assume that 
\begin{equation}
\gamma_1\ge \gamma_2\ge \gamma_3\ge\ldots.
\end{equation}
Furthermore, we fix a parameter $\alpha\in (1,\infty)$. For $k\in\NN_0$ we consider
\begin{align*}
r_{\alpha,\gamma_j}(k)=\begin{cases}1&\textnormal{if } k=0,\\\gamma_j k^{-\alpha} & \textnormal{if } k\neq 0.\end{cases}
\end{align*}
For a vector $\bsk=(k_1,\ldots,k_s)\in\NN_0^s$ we consider  
\begin{align*}
r_{s,\alpha,\bsgamma}(\bsk)= \prod_{j=1}^{s}r_{\alpha,\gamma_j}(k_j).
\end{align*}
Clearly, it holds that $r_{s,\alpha,\bsgamma}$ is summable. From now on, we use the following notation for the kernel function, 
\begin{align*}
K_{s,\alpha,\bsgamma} (\bsx,\bsy):=\sum_{\bsk\in\NN_0^s} r_{s,\alpha,\bsgamma}(\bsk) H_{\bsk}(\bsx) H_{\bsk} (\bsy),
\end{align*}
to stress that the reproducing kernel depends on $\alpha$ as well as on the weight sequence $\bsgamma$. The corresponding reproducing kernel Hilbert space is then given by $\cH(K_{s,\alpha,\bsgamma})$. This choice of $r$ now decreases polynomially fast as $\bsk$ grows, which influences the smoothness of the elements in $\cH(K_{s,\alpha,\bsgamma})$. In \cite{IL14} it is shown that the smoothness parameter $\alpha$ is related to the differentiability of the functions which makes it reasonable to call $\cH(K_{s,\alpha,\bsgamma})$ a Hermite space of finite smoothness.

\subsection{Hermite spaces of analytic functions}

Let $\bsa=\{a_j\}_{j\in\NN}$ and $\bsb=\{b_j\}_{j\in\NN}$ be two weight sequences of real numbers, where we assume that $a_0:=\inf_j a_j>0$ and $b_0:=\inf_j b_j\ge 1$.
Moreover, we fix an $\omega\in(0,1)$ and for $\bsk\in\NN_0^s$ we define
\begin{align}
r_{s,{\omega,\bsa,\bsb}}(\bsk)=\omega^{\vert \bsk\vert_{\bsa,\bsb}}:=\omega^{\sum_{j=1}^{s}a_j{k_j}^{b_j}}=\prod_{j=1}^{s}\omega^{a_j{k_j}^{b_j}}.
\end{align}
We denote the reproducing kernel function by
\begin{align*}
K_{s,\omega,\bsa,\bsb}(\bsx,\bsy)=\sum_{\bsk\in\NN_0^d}\omega^{\vert \bsk\vert_{\bsa,\bsb}}H_{\bsk}(\bsx)H_{\bsk}(\bsy)
\end{align*}
to indicate again the dependence on the weights. The corresponding Hermite space is then given by $\cH(K_{s,\omega,\bsa,\bsb})$. With the choice of $r_{s,\omega,\bsa,\bsb}$ it follows that the functions in $\cH(K_{s,\omega,\bsa,\bsb})$ have exponentially decaying Hermite coefficients. Furthermore, this exponential decay guarantees that the functions are extremely smooth, in fact analytic, see \cite{IKLP14}.

\section{Tractability of Monte Carlo integration}\label{sectrac}

Now we study Monte Carlo integration in a Hermite space $\cH(K_r)$. For that we consider MC integration rules which are randomized linear algorithms of the form
\begin{align*}
\mathrm{MC}_{n,s}(\bsx_1,\ldots,\bsx_n;f)=\frac{1}{n}\sum_{i=1}^{n}f(\bsx_i),
\end{align*}
with independent and standard normal distributed random variables $\bsx_1,\ldots,\bsx_n$. In this setting we are interested in the randomized error of a MC algorithm which is given by
\begin{align*}
e^{\mathrm{MC}}(n,s)=\sup_{f\in\cH(K_r),\|f\|_{K_r}\leq 1}\E\left(\left\vert I_s(f)-\mathrm{MC}_{n,s}(\bsx_1,\ldots,\bsx_n;f)\right\vert^2\right)^{\frac{1}{2}},
\end{align*}
where the expectation is taken with respect to independent and identically distributed random variables $\bsx_1,\ldots,\bsx_n$. Furthermore, we consider the minimal number of function evaluations which is needed to reduce the initial error by a factor of $\varepsilon\in(0,1)$, i.e.,
\begin{align*}
n^{\mathrm{MC}}(\varepsilon,s)=\min\{n : e^{\mathrm{MC}}(n,s)\leq \varepsilon\}.
\end{align*}
Note that the initial error is $1$. We want to know how $n^{\mathrm{MC}}(\varepsilon,s)$ depends on $\varepsilon^{-1}$ and $s$. For that we study the tractability of MC algorithms where we follow the notions given in \cite{NW08}. We say that we have:
\begin{enumerate}
	\item \emph{Weak MC-tractability}, if
	\begin{align*}
	\lim_{s+\varepsilon^{-1}\to\infty}\frac{\log(n^{\mathrm{MC}}(\varepsilon,s))}{s+\varepsilon^{-1}}=0.
	\end{align*}
	\item \emph{Polynomial MC-tractability}, if there exist $c,p,q\in\RR^{+}$ such that
	\begin{align*}
	n^{\mathrm{MC}}(\varepsilon,s)\leq c\, s^{q}\, \varepsilon^{-p}\qquad\textnormal{for all }\quad s\in\NN,\ \varepsilon\in(0,1).
	\end{align*}
	\item \emph{Strong polynomial MC-tractability}, if there exist $c,p\in\RR^{+}$ such that
	\begin{align*}
	n^{\mathrm{MC}}(\varepsilon,s)\leq c\, \varepsilon^{-p}\qquad\textnormal{for all }\quad s\in\NN,\ \varepsilon\in(0,1).
	\end{align*}
The infimum of $p$ for which strong polynomial MC-tractability holds is called $\varepsilon$-exponent.

\end{enumerate}

With weak MC-tractability we rule out that the smallest number of function evaluations needed to achieve an $\varepsilon$-approximation depends exponentially on $\varepsilon^{-1}$ and $s$. Polynomial MC-tractability means that $n^{\mathrm{MC}}(\varepsilon,s)$ is bounded polynomially in $\varepsilon^{-1}$ and $s$. In the case of strong polynomial MC-tractability the upper bound is a polynomial in $\varepsilon^{-1}$ and independent of the dimension $s$.

First we derive a formula for the randomized error where we see that the error depends on the number of integration nodes by a factor of $1/\sqrt{n}$. This coincides with the convergence rate of MC algorithms of $\mathcal{O}(n^{1/2})$.

\begin{thm}\label{th:randomizederror}
For the randomized error of MC integration in the Hermite space $\cH(K_r)$ it holds that
\begin{align*}
e^{\mathrm{MC}}(n,s)=\frac{1}{\sqrt{n}}\left(\max_{\bsk\in\NN_0^s\backslash\{\bszero\}}r(\bsk)\right)^{\frac{1}{2}}.
\end{align*}
\end{thm}
\begin{proof}
We know for the randomized error that
\begin{align*}
e^{\mathrm{MC}}(n,s)=\frac{1}{\sqrt{n}}\sup_{f\in\cH(K_r),\|f\|_{K_r}\leq 1}\left(I_s(f^2)-I_s(f)^2\right)^{\frac{1}{2}},
\end{align*}
see, e.g., \cite[Theorem 1.1]{N92}. Moreover, by Parseval's identity,
\begin{align*}
I_s(f^2)=\int_{\RR^s}f(\bsx)^2\varphi_s(\bsx)d\bsx=\sum_{\bsk\in\NN_0^s}\hat{f}(\bsk)^2
\end{align*}
and 
\begin{align*}
I_s(f)=\int_{\RR^d}f(\bsx)\varphi_s(\bsx)d\bsx=\widehat{f}(\bszero).
\end{align*}
Hence,
\begin{align}
I_s(f^2)-I_s(f)^2&=\sum_{\bsk\in\NN_0^s\backslash\{\bszero\}}\hat{f}(\bsk)^2\nonumber\\
&=\sum_{\bsk\in\NN_0^s\backslash\{\bszero\}}r(\bsk)^{-1}\hat{f}(\bsk)^2\,r(\bsk)\nonumber\\
&\leq\|f\|_{K_r}^2\max_{\bsk\in\NN_0^s\backslash\{\bszero\}}r(\bsk)\label{eq:mcinequality}
\end{align}
Now we set $\bsk^*=\mathrm{arg\,max}_{\bsk\in\NN_0^s\backslash\{\bszero\}}r(\bsk)$ and we consider the special integrand $f(\bsx)=H_{\bsk^*}(\bsx)$. Then we get for the $\bsk$-th Hermite coefficient of $f$,
\begin{align*}
\hat{f}(\bsk)=\begin{cases}1&\textnormal{if }\bsk=\bsk^*\\0&\textnormal{otherwise}\end{cases}.
\end{align*}
Thus, we have that \eqref{eq:mcinequality} is fulfilled with equality for this choice of $f$. Hence, it follows that
\begin{align*}
e^{\mathrm{MC}}(n,s)=\frac{1}{\sqrt{n}}\left(\max_{\bsk\in\NN_0^s\backslash\{\bszero\}}r(\bsk)\right)^{\frac{1}{2}}.
\end{align*}
\end{proof}

\subsection{Tractability in Hermite spaces of finite smoothness}

Now we consider MC-tractability of multivariate integration in Hermite spaces $\cH(K_{s,\alpha,\bsgamma})$ of functions of finite smoothness. Since
\begin{align*}
\max_{\bsk\in\NN_0^s\backslash\{\bszero\}}r_{s,\alpha,\bsgamma}(\bsk)=\max_{k=1,\ldots,s}\prod_{j=1}^{k}\gamma_j,
\end{align*}
we get from Theorem \ref{th:randomizederror} that the randomized error of MC integration is given by
\begin{align}\label{eq:randomizederrorpoly}
e^{\mathrm{MC}}(n,s)=\frac{1}{\sqrt{n}}\,\left(\max_{k=1,\ldots,s}\prod_{j=1}^{k}\gamma_j\right)^\frac{1}{2}.
\end{align}
We remark that this result is similar to the result in \cite{sw01} and therefore we proceed in the same way to study MC-tractability for the integration problem. 

From \eqref{eq:randomizederrorpoly} we see that MC integration is strongly polynomially MC-tractable if and only if $\sup_{s\in\NN}\prod_{j=1}^{s}\gamma_j<\infty$. Assume there exists a $j$ with $\gamma_j<1$, then we have that $\gamma_i<1$ for all $i\geq j$. Now let $j_0$ the smallest index such that $\gamma_{j_0}<1$. Then $\sup_{s\in\NN}\prod_{j=1}^{s}\gamma_j<\infty$ is equivalent to $\prod_{j=1}^{j_0-1}\gamma_j<\infty$. On the other hand, if $\gamma_j\geq1$ for all $j\in\NN$, we have that $\sup_{s\in\NN}\prod_{j=1}^{s}\gamma_j<\infty$ iff $\prod_{j=1}^{\infty}\gamma_j<\infty$. Altogether, we have that $\sup_{s\in\NN}\prod_{j=1}^{s}\gamma_j<\infty$ is equivalent to $\prod_{j=1}^{\infty}\max(\gamma_j,1)<\infty$ which, in turn, is equivalent to $\sum_{j=1}^\infty\max(\log(\gamma_j),0)<\infty$.  

Furthermore, we see from \eqref{eq:randomizederrorpoly} that we have polynomial MC-trac\-ta\-bility iff there exist $C,q>0$ such that $\max_{k=1,\ldots,s}\prod_{j=1}^{k}\gamma_j\leq C s^{q}$. As above, we get that this is equivalent to $\sup_{s\in\NN}\sum_{j=1}^{s}\max(\log(\gamma_j),0))/\log(s)<\infty$.

Finally, we again conclude from \eqref{eq:randomizederrorpoly} that integration is weakly MC-tratable if and only if $\max_{k=1,\ldots,s}\sum_{j=1}^{k}\log(\gamma_j)/s$ approaches zero as $s$ goes to $\infty$. Again this is equivalent to $\lim_{s\to\infty}\sum_{j=1}^{\infty}\max(\log(\gamma_j),0)/s=0$. Now we summarize our results in the next theorem.

\begin{thm}\label{th:mctractability}
MC integration in the weighted Hermite space $\cH(K_{s,\alpha,\bsgamma})$ is 
\begin{enumerate}
	\item strongly polynomially MC-tractable iff $\sum_{j=1}^{\infty}\max(\log(\gamma_j),0))<\infty$,
	\item\label{it:polymc} polynomially MC-tractable iff $A:=\limsup_{s\to\infty}\sum_{j=1}^{s}\frac{\max(\log(\gamma_j),0)}{\log(s)}<\infty$,
	\item weakly MC-tractable iff $\lim_{s\to\infty}\sum_{j=1}^{s}\frac{\max(\log(\gamma_j),0)}{s}=0$.
\end{enumerate}
\end{thm}

Let us give some remarks on Theorem \ref{th:mctractability}. We see that the conditions are necessary and sufficient. Moreover, these conditions are fulfilled, if the weight sequence contains weights which are smaller or equal than $1$. Especially, in the case of the unweighted Hermite space, i.e., $\gamma_j=1$ for all $j\in\NN$, we can achieve these three notions of MC-tractability using randomized linear algorithm. 

Note that, if we have strong polynomial MC-tractability, then the $\varepsilon$-exponent is $2$. Furthermore, the minimal number $n^{\mathrm{MC}}(\varepsilon,s)$ of function evaluations which is needed to guarantee that the randomized error is smaller than $\varepsilon$ is
\begin{align*}
\sup_{s\in\NN}n^{\mathrm{MC}}(\varepsilon,s)=C \varepsilon^{-2}
\end{align*}
with $C=\sup_{s\in\NN}\prod_{j=1}^{s}\gamma_j<\infty$. If we have polynomial MC-tractability, then 
\begin{align*}
n^{\mathrm{MC}}(\varepsilon,s)\leq s^{A+o(1)}\varepsilon^{-2}\qquad\textnormal{as } s\longrightarrow \infty
\end{align*}
with $A$ as in Theorem \ref{th:mctractability}. Furthermore, we remark that the conditions on MC-tractability of multivariate integration in Hermite spaces of finite smoothness are the same as for Monte Carlo integration in Korobov spaces, see \cite{sw01}, and in Walsh spaces, see \cite{dp05}.

\subsection{Tractability in Hermite spaces of analytic functions}

For the Hermite space $\cH(K_{s,\omega,\bsa,\bsb})$ of analytic functions we have that
\begin{align*}
\max_{\bsk\in\NN_0^s\backslash\{\bszero\}}r_{s,\omega,\bsa,\bsb}(\bsk)=\max_{\bsk\in\NN_0^s\backslash\{\bszero\}}\prod_{j=1}^{d}\omega^{a_j{k_j}^{b_j}}=\omega^{a_0}<\infty,
\end{align*}
because $a_0=\inf_j a_j>0$ and $b_j\geq 1$ for all $j\in\NN$. From Theorem \ref{th:randomizederror} we get that
\begin{align}\label{eq:randomizederrorexp}
e^{\mathrm{MC}}(n,s)=\frac{\omega^{a_0}}{\sqrt{n}}
\end{align}
and it is easy to see that we can achieve MC-tractability independent of the choice of the weight sequences $\bsa$ and $\bsb$.

\begin{thm}
MC integration in the weighted Hermite space $\cH(K_{s,\omega,\bsa,\bsb})$ is strongly polynomially MC-tractable, polynomially MC-tractable and weakly MC-tractable for all $\bsa$ and $\bsb$.
\end{thm}

In the worst case setting it is natural to expect exponential convergence for studying multivariate integration in the Hermite space of analytic functions, see \cite{IKLP14}. From Theorem \ref{th:randomizederror} it follows that we can not achieve 
exponential convergence for the error in the randomized setting by using standard Monte Carlo integration, but maybe it could be done by more sophisticated randomized algorithms. 

Furthermore, in \cite{IKLP14} notions of tractability are considered to study the dependence of $n^{\mathrm{MC}}$ on $s$ and $\log\varepsilon^{-1}$. If
\begin{align}\label{eq:ecwt}
\lim_{s+\varepsilon^{-1}\to\infty}\frac{\log( n^{\mathrm{MC}}(\varepsilon,s))}{s+\log\varepsilon^{-1}}=0
\end{align}
with $\log 0=0$ taken by convention, it is ruled out that the minimal number $n^{\mathrm{MC}}$ of function evaluations to achieve an $\varepsilon$-approximation to the initial error depends exponentially on $s$ and $\log\varepsilon^{-1}$. However, \eqref{eq:ecwt} cannot hold for Monte Carlo integration, because 
\begin{align*}
n^{\mathrm{MC}}(\varepsilon,s)=\left\lceil \varepsilon^{-2} \omega^{2a_0}\right\rceil.
\end{align*}
We remark that it is possible in the worst case setting to achieve better convergence rates and related notions of tractability, if we restrict ourself to function spaces of analytic functions, see \cite{IKLP14}. However, this is not possible in the randomized setting using Monte Carlo integration as we have seen in this section.

\end{document}